\documentclass{amsart}
\usepackage{amssymb, latexsym, amsmath, amscd, epsfig, color}

\theoremstyle{plain}
\newtheorem{theorem}{Theorem}[section]

\newtheorem{corollary}[theorem]{Corollary}
\newtheorem{proposition}[theorem]{Proposition}

\theoremstyle{definition}

\def\e{\varepsilon}

\def\vf{\varphi}
\def\a{\alpha}

\def\d{\delta}

\def\g{\gamma}

\def\o{\omega}

\def\l{\lambda}

\def\s{\sigma}

\def \RR{\mathbb R}
\def \N{{\mathbb N}}

\def \Z{\mathbb Z}

\def \P{\mathcal P}
\def\ov{\overline}

\def\om{\omega}

\def\K{\mathcal K}

\def\P{\mathcal P}

\def\QQ{\mathcal Q}
\def\LL{\mathcal L}

\def\({\biggl(}
\def\){\biggr)}

\def\<{\bold\langle}
\def\>{\bold\rangle}

\def\LL{{{\mathcal L}}}

\begin{document}

\title[Regularity]{Regularity of the drift and entropy of random walks on groups}

\author{Lorenz Gilch}
\email{gilch@TUGraz.at} 
\author{Fran\c cois Ledrappier}
  \email{ledrappier.1@nd.edu}
\address{Institut f\"ur Mathematische Strukturtheorie \\ Graz University of
  Technology\\ Steyrergasse 30/III\\ 8010 Graz, Austria}
\address{Department of Mathematics \\University of Notre Dame\\ Notre Dame, IN 46556, USA}

\maketitle

Random walks on a group $G$ model many natural phenomena. A random walk is
defined by a probability measure $p$ on $G$. We are interested in  asymptotic
properties of the random walks and in particular in the linear drift and the
asymptotic entropy. If the geometry of the group  is rich, then these numbers
are both positive and the way of dependence on $p$ is itself a property of $G$. In this note, we review recent results about the regularity of  the drift and the entropy for free groups, free products and hyperbolic groups. 

\

\section{Entropy and linear drift} 

We recall in this section the main notations for the objects under
consideration associated to a
group $G$ and a probability measure $p$ on $G$. Background on random walks can be found in the survey papers \cite{KV} and \cite{V} and in the book by W. Woess (\cite{W}).
\par
Let $G$ be a finitely generated group and $S$ a symmetric finite generating set. For
$g \in G$, let $|g|$ denote the smallest $n\in\mathbb{N}$ such that $g$ can be written
as $g = s_1 \cdots s_n$, where $s_1,\dots,s_n\in S$. We denote by $d(g,h):=
|g^{-1}h| $ the left invariant associated metric. Let $p$ be a probability
measure on $G$ with support $B$. Unless otherwise specified, we always assume
that $B$ is finite and that  $\bigcup_{n\in\mathbb{N}} B^n = G$. We denote by $\P (B)$ the set of probability measures with support $B$. The set $\P(B)$ is naturally identified with an open subset of the probabilities on $B,$ which is  a contractible open polygonal bounded convex domain in $\RR^{|B|- 1}$. We form, with $ p^{(0)} $ being the Dirac measure at the identity $e$,
$$ p^{(n)} (g) \; = \; [p^{(n-1)} \star p] (g) \; = \;  \sum _{h \in G}
p^{(n-1)} (gh^{-1}) p(h),$$
where $g\in G$. Let $\varrho (p) $ be the spectral radius of the Markov operator on $\ell _2(G)$ defined by $p$. It is given by
$\varrho(p)=\limsup_{n\to\infty} p^{(n)}(e)^{1/n}$.
Define   the entropy $H_{n,p}$ and the drift $L_{n,p}$ of $p^{(n)} $ by:
$$ H_{n,p} := -\sum _{g\in G} p^{(n)}(g) \ln p^{(n)}(g), \quad L_{n,p} := \sum _{g\in G} |g| p^{(n)}(g), $$
and the asymptotic entropy $h_p$  and the  linear drift $\ell _p $ by
$$ h_p := \lim_{n\to\infty} \frac{1}{n}H_{n,p}, \quad \ell _p := \lim_{n\to\infty} \frac{1}{n} L_{n,p}. $$

Both limits exist by subadditivity and Fekete's Lemma. The linear drift makes
sense as soon as $\sum_{g\in G} |g| p(g)  < +\infty $, the entropy under the
slightly weaker condition $H_{1,p} < +\infty .$ The entropy $h_p$ was
introduced by Avez (\cite{Av}) and is related to bounded solutions of the
equation on $G$ of the form $f(g) = \sum_{h \in G} f(gh) p(h)$ (see
e.g. \cite{KV}). In particular, $h_p = 0 $ if and only if the only bounded
solutions (called \textit{harmonic} functions) are the constant functions (\cite{KV}, \cite{De2}). A general relation is (\cite{Gu}) \begin{equation}\label{fundamental} h_p\;  \leq \;  \ell _p v,\end{equation}
where $v := \lim_{n\to\infty} \frac{1}{n} \ln \left(\# \{g\in G ; |g| \leq n
  \}\right) $ is the {\it {volume entropy}} of $G$. In particular, if $\ell
_p = 0 $ then $h_p = 0 $. 

We say that $p$ is symmetric if $B = B^{-1} $ and $p(g) = p(g^{-1} ) $ for all
$g \in B$. We call $p$ centered if $\sum_{g\in B} \chi(g) p(g) = 0 $ for all
group morphisms $\chi:G\to\RR$. Clearly, symmetric probabilities are
centered. If $p$ is centered and $h_p = 0$, then $\ell _p = 0$ (\cite{Va},
\cite{Mat}). If $p$ is not centered, we may have $h_p = 0 $ and $\ell _p \not =
0 $, for instance on $\Z$. If this is the case, there is a group morphism
$\chi:G\to\RR$ such that $\ell _p = \sum _{g\in B} \chi (g) p(g)$ (\cite{KL}, see also \cite {EKn} for finite versions of this result). 

\

Both $h_p $ and $\ell _p $ describe asymptotic properties of the random
    walk directed by $p$. Let $(\Omega , P) = (G^\N, p^{\otimes \N})$ be the
infinite product space so that    $ \o = (\o_1,\o_2,\ldots )\in G^\N$ is
realized by a sequence
of i.i.d. random variables with values in $G$ and distribution $p$. We form the right random walk by
$X_n (\o) := \o_1\o_2 \cdots \o_n.$ The probability $p^{(n)} $ is the
distribution of $X_n $, and an application of Kingman's subadditive ergodic
theorem (\cite{Ki}) gives that, for $P$-a.e. $\o$, 
\begin{equation}\label{kingman}\lim _{n\to\infty}  \frac{1}{n} |X_n| \; = \; \ell _p \quad {\textrm {and}} \quad \lim _{n\to\infty}  -\frac{1}{n} \ln \big( p^{(n)} (X_n) \big) \; = \; h_p. \end{equation}

The random walk is said to be recurrent if, for $P$-a.e. $\o$,  there is  a
positive $n\in\mathbb{N}$ with $X_n (\o) = e$. In this case there is an
infinite number of integers $n$ with $X_n = e$ and, by (\ref{kingman}), $\ell_p = 0
$. Hence, $h_p = 0 $. From here on, we  assume that the random walk is
transient, i.e. $|X_n| \to \infty $ for $P$-a.e. $\o$. The \textit{Green function}
$G(g,h)$, $g,h\in G$, is defined by 
$$ G(g,h) \; := \; \sum _{n \geq 0 } p^{(n)} (g^{-1}h).$$ 
By decomposing of the first visit to $h$ and using transitivity of the random walk we get 
$$
G(g,h) \; = \; F(g,h) G(h,h) \; = \; F(g,h) G(e,e) , 
$$
where $F(g,h)=\sum_{n\geq 1} f^{(n)}(g,h) $ with $f^{(n)}(g,h)$ being the
probability of reaching $h$ for the first time after $n$ steps when starting at $g$. If $p$ is
symmetric, then the (left invariant) Green distance is defined by $d_G(g,h)  := - \ln F(g,h) $. The drift $\ell _{p,G} $ for that distance  coincides with the entropy $h_p$  (\cite{BP}, Proposition 6.2, \cite{BHM1}) and the volume entropy is 1, so that there is equality in (\ref{fundamental}) for that distance (\cite{BHM1}).

\

We now turn to another representation of the drift and entropy.
Let $X$ be a compact space. $X$ is called a $G$-space if the group $G$ acts by homeomorphisms on $X$. This action extends naturally to probability measures on $X$. We say that the measure $\nu $ on $X$ is stationary if $\sum _{g \in G}( g_\ast \nu ) p(g) = \nu .$ The {\it {entropy}} of a stationary measure $\nu $ is defined by
\begin{equation}\label{entr} h_p (X,\nu ) \; := \; - \sum _{g \in G} \left( \int _{X}\ln  \frac {dg^{-1}_\ast \nu }{d\nu } (\xi) d\nu(\xi)\right) p(g). \end{equation}
The entropy $h_p$ and the linear drift $\ell _p$ are given by variational formulas  over stationary measures (see \cite {KV} for the entropy, \cite {KL} for the linear drift): 
\begin{eqnarray}\label{entropy}
h_p &=& \max \{ h_p (X ,\nu ); X \; G{\textrm {-space and }} \nu {\textrm { stationary on }} X \},\\ \label{escape} \ell _p & = &\max \big\{\sum _{g \in G}\left( \int _{\ov G} \xi(g^{-1}) d\nu (\xi)\right) p(g); \nu {\textrm { stationary on }} \ov G \big\},
\end{eqnarray}
where $\ov G$ is the Busemann compactification of $G$, the elements of which are horofunctions $\xi $ on $G$. A pair $(X,\nu )$, where $X$ is a $G$-space  and $\nu $ a $p$-stationary measure is called a {\it { boundary }} if, for $P$-a.e.$\om $,  $(X_n(\om ) )_\ast \nu $ converge towards a Dirac measure. It is called a {\it {Poisson boundary }} if it is a boundary and it realizes the maximum in formula (\ref{entropy}).

\

From the definition of $\ell _p$ and $h_p$, one sees that the mappings $p
\mapsto \ell_p$ and $p\mapsto h_p$ are uppersemicontinuous on $\P(B)$. Erschler (\cite{E}) raised the question of continuity of these functions and gave examples where these mappings  are not continuous on the closure of $\P(B)$. The question of continuity in general on the interior of $\P(B)$ is open. In the rest of the paper, we discuss several examples where one can prove stronger regularity results.

\

\section{Nearest neighbour random walks on a free group}

In the case when the group $G$ is a \textit{free group} with $d$ generators, $d \geq 2$, and $p$ is supported by these generators, explicit computations can be made  (see \cite{DM}).

Let $G$ be the free group with set of generators $S = \{ \pm i; i = 1,\dots, d
\}, $ where  $-i = i^{-1} $ for $i\in  S$. Let  $\P(S) $ be the set of
probability measures on $G$ with support $S$.
Since $d \geq 2$, as $n$ goes to infinity, the reduced word representing
$X_n(\o)$ converges towards an infinite reduced word $X_\infty (\o ) = s_1(\o )
s_2 (\o) \cdots $ with $s_j (\o ) \not = -s_{j+1} (\o) .$
Denote by $G_\infty$ the space of infinite reduced words. The stationary
measure is unique: it is the distribution $\nu $ of $X_\infty (\o)$. Then $(G_\infty,\nu)$ is  the Poisson boundary and $G_\infty $ is  the Busemann boundary of $G$. 
Let $q_i = P\bigl( \{ \o ;s_1(\o ) = i \}\bigr) = \nu ([i])$, where $[i]$ consists of all
infinite words in $G_\infty$ starting with letter $i\in S$. We have $\sum _{i
  \in S} q_i = 1$. Let $i_1\dots
i_k$ be a reduced word in $G$. Then $\nu$ is uniquely
determined by the values $\nu([i_1\dots i_k])=F(e,i_1\dots i_k)(1-q_{-i_k})$, where
$F(e,i_1\dots i_k)$ is the probability of hitting $i_1\dots i_k$ when starting at the
identity $e$.
 Formula  (\ref{escape}) writes:
$$ \ell _p = 1 - 2 \sum _{i\in S} p_i q_{-i} ,$$
where $p_i = p(i) .$ In order to write the formula for the entropy, we introduce
$z_i : = F(e, i)$ for $i \in S$. The density $\displaystyle  \frac {dg^{-1}_\ast \nu }{d\nu }
(\xi)$ gives the minimal positive harmonic function with pole at
$\xi=i_1i_2\ldots\in G_\infty$. The Green function satisfies the following multiplicative
structure:
$$
G(e,i_1\dots i_k)=F(e,i_1)G(i_1,i_1\dots i_k)=F(e,i_1)G(e,i_2\dots i_k).
$$
This yields together with \cite[Theorem 2.10]{Le}
$$  
\frac {di^{-1}_\ast \nu }{d\nu } (\xi) \; = \; 
\lim_{k\to\infty} \frac{G(-i,i_1\dots i_k)}{G(e,i_1\dots i_k)}
=
\begin{cases}
z_{i}, & \textrm { if } i_1 \not = -i,\\
z_{-i}^{-1}, & \textrm { if } i_1 = -i.
\end{cases}
$$
 Formula  (\ref{entropy}) writes:
$$ 
h_p = \sum _{i \in S} p_i \big[ q_{-i} \ln z _{-i} - (1-q_{-i} )\ln z _i
 \big].
$$

\

 We can express the $q_i$ in terms of the $p_i$, and vice versa, thanks to the
 {\it {traffic equations}}: using the Markov property, we can write:
$$ z_i \; = \; p_i + z_{i} \sum _{j \in S \setminus\{i\} } p_j z _{-j}  \quad
{\textrm {and}} \quad  q_i \; = \; z_i (1- q_{-i}) . $$ 
The last equation arises from the formula given for $\nu([i])$ above.
Setting $Y := \sum _{j\in S} p_j z_{-j}$, we get 
$$ p_i = \frac {z_i (1-Y)}{1- z_i z_{-i} } \quad {\textrm {and}} \quad z_i \; = \; \frac{q_i } {1-q_{-i}} ,$$ so that we find:
$$ \ell_p \; = \; 1- \frac{2}{A} \sum _{i\in S} \frac {q_iq_{-i}
  (1-q_i)}{1-q_i-q_{-i} },  \quad {\textrm {where}} \quad A \; = \;
(1-Y)^{-1}\;  = \; 1 +  \sum _{i\in S} \frac {q_iq_{-i}}{1-q_i -q_{-i}},$$
which writes:
$$ \ell _p = \frac {B}{A}, 
{\textrm { where } } B:= 1 - \sum_{i\in S}  \frac{q_i q_{-i}(1-2q_i)}{1-q_i-q_{-i}} .$$

Hence, in terms of the $q_i, i \in S$, $p_i$ and $\ell _p $ are rational,  and the expression of $h_p$ involves rational functions and $\ln q_i, \ln (1-q_i)$. 
\begin{proposition}\label{analfree} The mappings $p\mapsto \ell_p $ and $p\mapsto h_p $ are real analytic on $\P(S).$ \end{proposition}
\begin{proof} Since all formulas are explicit in terms of the
  $q_i$, we only have to check that the $q_i$ are  real analytic functions on
  $\P(S)$. First, since $z_i$ can be seen as a sum of probabilities of
  different paths in $G$
we can write $z_i$ as a power series in terms of the $p_i$'s
  and the additional variable $z\in\mathbb{C}$ (contained in a suitable domain
  of convergence),
  namely as
$$
z_i(z) = \sum_{(n_1,\dots,n_{2d})\in\mathbb{N}^{2d}} c(n_1,\dots,n_{2d}) 
p_1^{n_1}p_{-1}^{n_2}p_2^{n_3}p_{-2}^{n_4}\cdot\ldots \cdot p_d^{n_{2d-1}}p_{-d}^{n_{2d}} z^{n_1+\dots +n_{2d}},
$$
where $c(n_1,\dots,n_{2d})\geq 0$. Note that $z_i=z_i(1)$. Since the spectral
radius is strictly smaller than $1$ (see e.g. \cite[Cor. 12.5]{W}), the power
series $G(e,i|z)=\sum_{n\geq 0} p^{(n)}(i)z^n$ has radius of convergence
strictly bigger than $1$ and satisfies $G(e,i|z)\geq z_i(z)$ for all real
$z>0$. That is, for each $p\in\mathcal{P}(S)$, $z_i(z)$ has radius of
convergence $R_i>1$. Choose now any $\delta>0$ with
$1+\delta <R_i$. Then
\begin{eqnarray*}
z_i&= &z_i(1)\leq z_i(1+\delta) \\
&=& \sum_{(n_1,\dots,n_{2d})\in\mathbb{N}^{2d}} c(n_1,\dots,n_{2d}) 
\bigl((1+\delta)p_1\bigr)^{n_1}\cdot\ldots \cdot \bigl((1+\delta)p_{-d}\bigr)^{n_{2d}}<\infty.
\end{eqnarray*}
In other words, $z_i=z_i(1)$ is real analytic in a neighbourhood of any $p\in\mathcal{P}(S)$.
The equations
$q_i=z_i(1-q_{-i}), q_{-i}=z_{-i}(1-q_i)$
give
$$
q_i=\frac{z_i(1-z_{-i})}{1-z_iz_{-i}},
$$
and this finishes the proof.
\end{proof}

Observe that, for $d= 1$, the group $G$ is $\Z$, $ S = \{\pm 1\}$ and  $p \mapsto \ell _p = | p_1 - p_{-1}| $ is not a real  analytic function on $\P(\pm 1).$

\

The formulas are even simpler when the probability $p$ is symmetric. Let $\P_\s (S) $ be the set of symmetric probability measures on $S$; elements of $\P_\s(S)$ are described by $d$ positive numbers $\{p_1, \cdots , p_d \}$ such that $\sum _{i=1}^d p_i = 1/2.$ If $p \in \P_\s (S)$,  $q_i = q_{-i}$ and we have:
\begin{eqnarray*}  
\ell _p &=& \frac {B}{A} {\textrm { with }}  A = 1 +2  \sum
  _{i=1}^d \frac {q_i^2}{1-2q_i}  {\textrm { and }}  B= 1 - 2\sum_{i=1}^d
  q_i^2, \\  h_p &=& -\frac{2}{A} \sum_{i=1}^d q_i (1-q_i ) \ln \frac{q_i
  }{1-q_i} , {\textrm {whereas}} \\ p_i &=& \frac{q_i (1-q_i) } {A
    (1-2q_i)}\quad \textrm{ for }i=1,\dots,d.
\end{eqnarray*}

\begin{proposition}\label{maximum}  The functions $p \mapsto \ell _p$ and
  $p\mapsto h_p  $ reach their  maxima on $\P_\s(S)$ at the constant vector
  $p_0 = (1/2d, \ldots , 1/2d )  $  and 
$$
\ell _{p_0} = 1 - \frac{1}{d}, \quad h_{p_0} = \bigl(1-\frac{1}{d}\bigr) \ln
(2d-1).
$$
\end{proposition}
\begin{proof} 

 By symmetry, the constant vector $p_0$  is a critical point for $\ell _p$. At $p_0$, $q_i = 1/2d$ by symmetry and  $\ell_{p_0} = 1 - 1/d, h_{p_0}= (1-1/d )\ln (2d-1)$  by the formulas above (observe that these expressions are also valid for $d=1$:  the only point of $\P_\s (\pm 1) $ is $(1/2, 1/2)  $, for which $\ell = 0 = 1- 1/d$ and $h = 0 = (1-1/d) \ln (2d-1)$).  Moreover,  the volume entropy of $G$ is $\ln (2d -1)$. By (\ref{fundamental}), the result for $\ell _p$ implies that $(1-1/d) \ln (2d-1)$ is the maximal value that the entropy might take on $\P_\s(S)$. Since this  is the entropy $h_{p_0} $, $p_0$ achieves the maximum of the entropy as well.
 
 We are going to prove  that the function $\displaystyle (q_1, \ldots ,q_d) \mapsto  B/A $  has a unique critical point on the set $\{(q_1, \ldots ,q_d); q_j > 0, \sum _{j=1}^d q_j = 1/2 \} .$
Observe that the formula for $\ell _p $ is continuous on the domain  $0 \leq q_i \leq 1/2 $  and that  the value of $\ell _p $ at the boundary of the domain $\{(q_1, \ldots ,q_d); q_j > 0, \sum _{j=1}^d q_j = 1/2 \} $
 is  the one computed with only the non-zero $q_i$'s on a free group with a smaller set of generators. Since at the constant vector $p_0$, $\ell _{p_0} = 1- 1/d $, it follows, by induction on the dimension, that the critical point $p_0$ is a maximum. The proof for $d=2$ is the same as in the general case: there is only one critical point by the argument below and the limit of the expression for $\ell _p$ at $(0,1/2), (1/2, 0)  $ is $0$.

 Using a Lagrange multiplier, we are looking for the critical points of the function $F(q, \l ) = \ell _p - \l (\sum_{j=1}^d q_j -1/2)$ satisfying $0\leq q_j \leq 1/2$ for $j = 1, \ldots, d.$ Setting as above $$A = 1 +2  \sum _{i=1}^d \frac {q_i^2}{1-2q_i} \quad {\textrm {and}} \quad B= 1 - 2\sum_{i=1}^d q_i^2 , $$
all equations $\frac{\partial F}{\partial q_i} = 0 $ depend only on $A,B, \l $ and $q_i$. 

Indeed,  they write $G(A,B,\l, q_i) = 0 $, where:
$$ G(A,B,\l, q) = 16 A q^3 + 4q^2 ( \l A^2 -4 A-B) +4q ( -\l A^2 +A +B) +\l A^2 . $$
If, for fixed $A,B,\l$, the equation  $G(A,B,\l, q_i) = 0 $ has only one solution $q \in [0,1/2) $, then, for these values of $A,B,\l$, the only possible critical point of $F$ is $q_j = 1/2d $ for all $j$. Then, unless $A = \frac{2d-1}{2d-2}, B= \frac{2d-1}{2d}$, there is no critical point for $F$ with those values of $A,B$.

 To summarize,  we only have to verify that the equation  $G(A,B,\l, q_i) = 0 $ has at most  one solution $q \in [0,1/2)$ for all $A,B,\l$ with $0 < B < 1 < A $.

 The function  $q \mapsto  G(A,B,\l, q) $ is a third degree polynomial with positive highest coefficient, $1/2$ is a critical point and $G(A,B,\l, 1/2 ) = B > 0$. Therefore, there is at most one solution  $q \in [0,1/2) $. \end{proof}
Let us mention the recent article \cite[Corollary 2.7]{GMM}, where bounds for
$\ell_p$ and $h_p$ are provided which are equal to the values of $\ell_{p_0}$
and $h_{p_0}$ in our context of free groups.
 It is likely that $p_0$ gives also the maximum of the entropy on the whole $\P(S)$, but we do not have  a proof of that fact. We also conjecture that the
 mapping $p\mapsto \ell_p$ is a concave function; numerical computer calculations of the drift for
 small $d\in\mathbb{N}$ supports and confirms this conjecture, but we do not
 have a proof for general $d$.

\section{Free products, Artin dihedral groups  and braid groups}
 
 The computations in Section 2 have been known for fifty years (even if Proposition  \ref{maximum} seems to be formally new). There are few other examples where it is possible to describe geometrically the Poisson boundary and the Busemann boundary, and it is even rarer to be able to give useful formulas for the stationary measure. In this section, we review the examples we are aware of.
\par
One important concept of constructing new groups from given ones is the free
product of groups. The crucial point is that free products have a tree-like
structure. More precisely, suppose we are given finitely generated groups $G_1,\dots, G_r$ equipped with
finitely supported probability measures $p_1,\dots,p_r$. The identity of $G_i$
is denoted by $e_i$, and w.l.o.g. we assume that these groups are pairwise
disjoint and we exclude the case $r=2=|G_1|=|G_2|$ (this case leads to
recurrent random walks in our setting). The free product
$G_1\ast \dots \ast G_r$ is given by 
$$
G=\ast_{i=1}^r G_i=\bigl\lbrace
x_1x_2 \dots x_n \bigl| x_j\in\bigcup_{i=1}^r G_i\setminus\{e_i\}, x_j\in G_k
\Rightarrow x_{j+1}\notin G_k
\bigr\rbrace\cup \{e\},
$$
the set of finite words over the alphabet $\bigcup_{i=1}^r G_i\setminus\{e_i\}$
such that two consecutive letters do \textit{not} come from the same group
$G_k$, where $e$ describes the empty word. A group operation on $G$ is given by
concatenation of words with possible contractions and cancellations in the middle such that
one gets a reduced word as above. For $x=x_1\dots x_n\in G$, define the \textit{block
  length} of $x$ as $\Vert x\Vert:=n$. Let us mention that the free group with
$d$ generators is a free product $\mathbb{Z}\ast\dots \ast\mathbb{Z}$ with $k$
free factors $\mathbb{Z}$.
\par
A random walk on $G$ is constructed in a natural way as follows: we lift $p_i$
to a probability measure $\bar p_i$ on $G$: if
$x=x_1\dots x_n\in G$ with $x_n\notin G_i$ and $v,w\in G_i$,
then $\bar p_i(xv,xw):=p_i(v,w)$. Otherwise we set $\bar p_i(x,y):=0$. Choose
$0<\alpha_1,\dots,\alpha_r\in\mathbb{R}$ with $\sum_{i=1}^r \alpha_i =
1$. Then we obtain a new probability measure on $G$ defined by
$$
p=\sum_{i=1}^r \alpha_i \bar p_i
$$
with $B=\mathrm{supp}(p)=\bigcup_{i=1}^r \mathrm{supp}(p_i)$.
We consider random walks $(X_n)_{n\in\mathbb{N}_0}$ on $G$ starting at
$e$, which are governed by $p$. For $i\in\{1,\dots,r\}$, denote by $\xi_i$ the
probability of hitting the set $G_i\setminus\{e_i\}$ when starting at $e$. The
spectral radius $\varrho(p)$ is strictly less than $1$ due to the non-amenability of $G$.
Let $\partial G_i$ be the Martin boundary of $G_i$ with respect to $p_i$, and
denote by $G_\infty$ the set of infinite words $x_1x_2\dots$ such that $x_i\in
G_k$ implies $x_{i+1}\notin G_k$. Then the Martin boundary of $G$ is given by
$$
\partial G = G_\infty \cup \bigcup_{i=1}^r \{x \xi ; x=x_1\dots x_n\in G,x_n\notin
G_i,\xi\in \partial G_i\};
$$
see e.g. \cite[Proposition 26.21]{W}. The random walk on $G$ converges almost
surely to an infinite word in $G_\infty$ in the sense that the common prefix of
the word at time $n$ and the infinite limit word tends to infinity as $n\to\infty$. The limit distribution $\nu$ is determined
by
$$
\nu\bigl(\{x_1x_2\dots \in G_\infty ; x_1=y_1,\dots,x_n=y_n\})=F(e,y_1\dots y_n)\bigl(1-(1-\xi_i)G_i(\xi_i)\bigr),
$$
where $n\in\mathbb{N}$, $y_1\dots y_n\in G$ with $y_n\notin G_i$, $F(e,y_1\dots y_n)$ being the probability of hitting $y_1\dots y_n$ and
$G_i(z)=\sum_{n\geq 0} p_i^{(n)}(e_i)z^n$ with $z\in\mathbb{C}$; see e.g. \cite{Gi1}.
\par
The
next propositions summarize results about regularity of drift and
entropy. Explicit formulas can be found in the cited sources.
\begin{proposition}[\cite{Gi1}]\label{prop:freeproduct-drift}
The drift w.r.t. the block length $\ell_{B}  = \lim_{n\to\infty}
\frac{1}{n}\Vert X_n\Vert$ exists and varies real analytically
in $p\in \mathcal{P}(B)$.
\end{proposition}
\begin{proof}
In \cite[Equ. (9)]{Gi1} a formula for $\ell_{B}$ is given:
$$
\ell_B= 
\sum_{i=1}^r\alpha_i\,\frac{1-\xi_i}{\xi_i}\,\bigl(1-(1-\xi_i)\,G_i(\xi_i)\bigr).
$$
Let be $d=|B|-1$, and write $p=(q_1,\dots,q_d)\in\mathcal{P}(B)$.
Analogously to the proof of Proposition \ref{analfree} one can write $\xi_i$ as a power series
(evaluated at $z=1$) in the form
$$
\xi_i(z)=\sum_{(n_1,\dots,n_{d})\in\mathbb{N}^{d}} c(n_1,\dots,n_{d}) 
q_1^{n_1}q_2^{n_2}\dots q_d^{n_{d}} z^{n_1+\dots +n_{d}}, \quad z\in\mathbb{C}.
$$
Since $\varrho(p)<1$ the Green functions $G(g|z)=\sum_{n\geq 0}
p^{(n)}(g)z^n$, $g\in G$,
have radii of convergence $R=1/\varrho(p)>1$ and dominate $\xi_i(z)$ for real $z>0$. Hence,
$\xi_i(z)$ has radius of convergence bigger than $1$, which in turn --
following the same argumentation as in Proposition \ref{analfree} -- yields
real analyticity
of $\xi_i=\xi_i(1)$ in a neighbourhood of any $p\in\mathcal{P}(B)$. Furthermore,
$G_i(z)$ can be expanded in the same form as $\xi_i(z)$ and, for each real positive $z_0<1$,
the mapping $p\mapsto G_i(z_0)$ is real analytic. Since $\xi_i<1$
(see e.g. \cite[Lemma 2.3]{Gi1}) the
mapping $p\mapsto G_i(\xi_i)$ is also real analytic as a composition of real analytic functions. This yields the proposed statement.
\end{proof}

\begin{proposition}[\cite{Gi1}]\label{reg2}
Let $p$ govern a nearest neighbour random walk on $G$, that is, the length 
$|g|$ is computed with respect to the generator $B$. Then the drift function $p\mapsto
\ell_p=\lim_{n\to\infty} \frac{1}{n}|X_n|$ is real analytic.
\end{proposition}
\begin{proof}
By the formula for $\ell$ given in \cite[Section 7]{Gi1} we just have to check that the mapping
$$
p\mapsto \widetilde G_j(y,z):=\sum_{m,n\geq 0} \sum_{x\in G_j:|x|=m} p_j^{(n)}(e_j,x)y^nz^m
$$
is real analytic for all $y\in(0,1)$ and $z=1$. For a moment fix $y<1$ and choose $\delta>0$
small enough such that $y(1+2\delta)^2<1$. Since $p_j^{(n)}(e_j,x)>0$, $x\in G_j$
with $|x|=m$, implies $n\geq m$, we get
$$
\widetilde G_j\bigl(y,(1+2\delta)^2\bigr)= \sum_{m,n\geq 0} \sum_{x\in G_j:|x|=m} p_j^{(n)}(e_j,x)\bigl(y(1+\delta)^2\bigr)^n
\leq \frac{1}{1-y(1+2\delta)^2}<\infty.
$$
This yields $\frac{\partial}{\partial z}\widetilde
G\bigl(y,(1+\delta)^2\bigr)<\infty$. Since each term $p^{(n)}_j(e_j,x)$ can be
written as a polynomial 
$$
\sum_{\substack{(n_1,\dots,n_d)\in\mathbb{N}^d:\\ n_1+\ldots + n_d=n}}c(n_1,\dots,n_d) q_1^{n_1}\cdot \ldots
\cdot q_d^{n_d},
$$
where $p=(q_1,\dots,q_d)\in\mathcal{P}(B)$ and $c(n_1,\dots,n_d)\geq 0$, we rewrite $\frac{\partial}{\partial z}\widetilde
G\bigl(y,(1+\delta)^2\bigr)$ as
$$
\frac{1}{1+\delta}\sum_{m,n\geq 0}\sum_{x\in G_j:|x|=m} m
\bigl(p_j^{(n)}(e_j,x)(1+\delta)^n\bigr) \bigl(\xi_j(1+\delta)\bigr)^n.
$$
That is, $\sum_{m,n\geq 0}\sum_{x\in G_j:|x|=m} m
p_j^{(n)}(e_j,x)\xi_j^n$ is real analytic in $\mathcal{P}(B)$ as a composition
of real analytic functions, and this yields
the claim.
\end{proof}

Let us mention that -- in contrast to Proposition \ref{maximum} --
simple random walk (that is, $\mu$ is the equidistribution on $\mathrm{supp}(\mu)$) is not necessarily the fastest random walk. Namely, it can be
verified with the help of \textsc{Mathematica} that -- with $p_i$ describing the
simple random walk on $G_i$ -- the simple random walk on
$(\mathbb{Z}/3\mathbb{Z})\ast (\mathbb{Z}/2\mathbb{Z})$ (that is,
$\alpha_1=2/3,\alpha_2=1/3$) is slower than the random walk on this
free product with the parameters $\alpha_1=\alpha_2=1/2$.
\par
Furthermore, we have the following regularity result:

\begin{proposition}[\cite{Gi3}]
Assume that $h_i:=-\sum_{g\in G_i} p_i(g) \ln p_i(g)<\infty$ for all
 $i\in\{1,\dots,r\}$, that is, all random walks on the factors $G_i$ have
 finite single-step entropy. Then the mapping $p\mapsto h_p$ is real analytic.
\end{proposition}

We now turn to another class of groups whose Cayley graphs have a tree-like structure.
A group $G$ is called \textit{virtually free} if it has a free subgroup of
finite index. At this point we assume that $G$ has a free subgroup with at
least $d\geq 2$ generators; otherwise, $G$ is a finite extension of
$\mathbb{Z}$ where we either get recurrent random walks or non-regularity
points on $\mathbb{P}(B)$. It is well-known that virtually
free groups can be constructed from a finite number of
finite groups by iterated amalgamation and HNN extensions. Each element of $G$
can be written as $x_1\dots x_nh$, where $x_i\in\{\pm i;i=1,\dots d\}$ and $h$
being one of finitely many representatives for the different cosets. Suppose we
are given a weight or length function $l(\pm i)\in\mathbb{R}$ for $i\in\{1,\dots,d\}$. Then a natural
length function on $G$ is defined by $l(x_1\dots x_nh)=\sum_{j=1}^n l(x_j)$.
We have the following result:

\begin{proposition}[\cite{Gi2}]\label{prop-amalgam}
Let $G$ be a virtually free group. Let $p$ govern a finite range random walk on $G$. Then the
mapping $p\mapsto \lim_{n\to\infty} l(X_n)/n$ is real analytic.
\end{proposition}
\begin{proof}
Random walks on virtually free groups can be interpreted as a random walk on a
regular language in the sense of \cite{Gi2}.
The claim follows from the formula for $\lim_{n\to\infty} l(X_n)/n$, the drift
with respect to the length function $l$, given in \cite[Theorem
2.4]{Gi2}. Due to
non-amenability of $G$ we have again $\varrho(p)<1$. The rest follows
analogously as in the proofs of
Propositions \ref{analfree} and \ref{prop:freeproduct-drift}.
\end{proof}
For the special case $l$ being the natural word length the last proposition is also
covered by Corollary \ref{drift(BA)}.
\par
At this point we want to mention the article \cite{MM2}, which uses similar techniques to
establish statements about the drift of random walks on the braid group $B_3$  and on Artin groups of
dihedral type. Traffic equations are established, whose unique solutions lead to
formulas for the drift. For random walks on these groups there might occur transitions (when varying
the probability measures of constant support), where one has no
regularity. An explicit example for non-differentiability points is given on
the braid group $B_3$; at these points the random walk changes its behaviour
essentially in terms of how the limit words looks like.  
However, \cite{MM2} gives explicit formulas for the drift in terms of the
solutions of the traffic equations splitted up into different branches. By
methods similar to the above, one can show  that the drift is real analytic on 
each branch. Indeed,  solutions of the traffic equations can be written as
converging power series as in the proofs of Propositions \ref{prop:freeproduct-drift}  and \ref{reg2}.

 \section{Hyperbolic groups}
 
 A geodesic metric space is called {\it {hyperbolic}} if geodesic triangles are
 thin: there is $\d \geq 0$ such that each side of a geodesic triangle is
 contained in a $\d$-neighbourhood of the union of the other two sides. A
 finitely generated group is called hyperbolic if the Cayley graph defined by
 some finite symmetric generating set is hyperbolic. This property does not depend
 on the choice of a  generating set. Free  groups are hyperbolic, as are fundamental
 groups of compact manifolds of negative curvature, and small cancellation
 groups. See e.g. \cite{GH} for the main geometric properties of hyperbolic
 groups. The geometric boundary of a hyperbolic space is the space of
 equivalence classes of geodesic rays, where two geodesic rays are equivalent
 if they are at a bounded Hausdorff distance. The geometric boundary $\partial
 G$ of the Cayley graph of  a hyperbolic group $G$ is a compact $G$-space. It
 is endowed with the Gromov metric (see \cite{GH}). The mapping $\Phi : G \to
 \Z^G, \Phi (g) (h)= |h^{-1} g| - |g| $ is an injective mapping  such that $\Phi (G)$ is
 relatively compact for the product topology. The Busemann compactification
 $\overline G$ is the closure of $\Phi (G)$ in $\Z^G$. There is an equivariant
 homeomorphism $\pi : \overline G \setminus G \to \partial G$ (see
 e.g. \cite{WW}). The homeomorphism $\pi$ is finite-to-one (see
 e.g. \cite{CP}). Following \cite{Bj}, we say that $G$ with the generating set  $S$
 satisfies the \textit{basic assumption} (BA) if the homeomorphism $\pi $ is one-to-one. In this case, we write, for $\xi \in \partial G, h \in G$, $\xi (h)$ for the value at $h$ of the sequence $\pi ^{-1} \xi  \in \Z^G.$ Free groups and surface groups with their natural generators satisfy (BA). It is an open problem whether any hyperbolic group admits a symmetric generating set  with the property (BA).
 
 Let $p$ be a probability measure on $G$ with finite support. Then, there is a unique $p$-stationary probability measure $\nu _p $ on $\partial G$ and $(\partial G, \nu _p) $ is a Poisson boundary for $(G, p)$(\cite{An}, \cite{K} Theorem 7.6). If (BA) is satisfied, the measure $\nu _p$  is the unique stationary probability measure on the Busemann compactification and formulas (\ref{entropy}) and (\ref{escape}) write:
 \begin{equation}\label{BA}  h_p \; = \; - \sum _{g \in B} \left(\int _{\partial G} \ln  \frac {dg^{-1}_\ast \nu_p }{d\nu_p } (\xi) d\nu_p(\xi)\right) p(g), \;\; \ell_p \; = \;  \sum _{g \in B} \left(\int _{\partial G} \xi (g^{-1}) d\nu_p(\xi)\right) p(g). 
 \end{equation}
 
 \begin{proposition}\label{anal} Assume that $(G,S)$ is a non-elementary
   hyperbolic group and satisfies (BA). Let  $p \in \P(B)$, $\a$ be  small
   enough, and let $f$ be an $\a$-H\"older continuous function on $\partial G$.
   Then the mapping $p \mapsto \int_{\partial G} f(\xi ) d\nu _p (\xi) $ is real analytic on a neighbourhood of $p$ in $\P(B) .$ \end{proposition}
 \begin{proof} Let  $\K_\a $ be the space of $\a$-H\"older continuous functions
   on $\partial G$. The space $\K_\a$ is a Banach space with norm $ \|f\|_\a$, where 
 $$ \|f\|_\a \; = \; \max _{\xi\in \partial G} |f(\xi ) | +
 \sup_{\xi,\eta\in\partial G: \xi \not = \eta} \frac{|f(\xi ) - f(\eta )|} {(d(\xi, \eta))^\a} .$$
 For $p \in \P(B)$, let  $\QQ _p$ be the operator on $\K_\a$ defined by 
 $$ \QQ _pf (\xi ) \; = \; \sum _{g \in B} f(g^{-1} \xi ) p(g) .$$
 Clearly, the mapping $p \mapsto \QQ _p $ is real analytic from $\P(B)$ into $\LL (\K_\a ).$ If $G$ is not elementary and satisfies (BA), it can be shown (see \cite{Bj}, Lemma 4) that, for $\a $ small enough,   $f \mapsto \int f d\nu _p $ is an isolated eigenvector for the transposed operator $\QQ _p^\ast $ on the dual space $\K_\a^\ast.$ The proposition follows by a perturbation lemma. \end{proof}
 
 \begin{corollary}\label{drift(BA)}  Assume that $(G,S)$ is a non-elementary hyperbolic group and satisfies (BA). Then the mapping $p \mapsto \ell _p $ is real analytic. \end{corollary}
 Indeed, the function $\xi (g^{-1}) $ in formula (\ref{BA}) belongs to $\K_\a$ for all $\a$.
 Corollary \ref{drift(BA)} is due to \cite{L2} in the case of the free group. P. Mathieu (\cite{Ma2}) proved the  $C^1$ regularity and gave a formula for $\nabla _p \nu _p$ and $\nabla _p \ell _p$ in the symmetric case, to be compared with formulas for  linear response of dynamical systems (cf. \cite{R}).
 
 \
 
 The formula (\ref{BA}) for the entropy is valid in general, even without the
 (BA) hypothesis, but observe that the integrand $\vf _p (g, \xi) :=
 \displaystyle - \ln  \frac {dg^{-1}_\ast \nu_p }{d\nu_p } (\xi) $ is itself a
 function of $p$. To study this function, we use the description by A. Ancona
 (\cite{An}) of the Martin boundary of a random walk with finite support on a
 hyperbolic group. Recall that $F_p(g,h) $ is the probability of reaching $h$
 starting from $g$ in dependence of $p$.
 \begin{proposition}[\cite{An}]\label{martin} Assume that $G$ is hyperbolic and
   that $p$ has finite support. Then
 $$ 
\vf _p(g, \xi )  \; = \; \lim _{h \to \xi} \ln \frac
 {F_p(e,h)}{F_p(g^{-1},h)} \quad \textrm{ for all } g\in G,\xi\in\partial G.
$$
 \end{proposition}
 A consequence of the proof of Proposition \ref{martin} is that, for all $g\in G$, for $\a $ small enough $\vf _p (g, \xi )  \in \K_\a $ (see \cite{INO}). In the case of free groups, Proposition \ref{martin} goes back to Derriennic (\cite{De1}) and using his arguments one can prove:
 \begin{proposition}[\cite{L2}] If $G$ is a free group and $p$ has finite support $B$, there is $\a$ small enough that, for all $g \in B$,  the mapping $p \mapsto \vf _p $ is real analytic from   a neighbourhood of $p$ in $\P(B)$ into $\K_\a$. \end{proposition}
 \begin{corollary} [\cite{L2}] If $G$ is a free group and  $p$ has finite support $B$,  the mapping $p \mapsto h_p $ is real analytic on $\P(B)$. \end{corollary}
 
 \
 
For cocompact Fuchsian groups there is the following recent result, to be compared with  Corollary \ref{drift(BA)}:
\begin{proposition}[\cite{HMM}] 
Let $G$ be a cocompact Fuchsian group with planar presentation. Then the
mapping $p\mapsto \ell_p$ is real analytic.
\end{proposition}

 For a general hyperbolic group, we have a weaker result:
  \begin{proposition}[\cite{L3}] If $G$ is a hyperbolic group and  $p$ has
    finite support $B$, there is $\a$ small enough such that, for all $g \in B$,
    the mapping $p \mapsto \vf _p $ is Lipschitz continuous from   a
    neighbourhood of $p$ in $\P(B)$ into $\K_\a$. \end{proposition} 
\begin{corollary} [\cite{L3}] 
If $G$ is a hyperbolic group and  $p$ has finite
  support $B$,  the mappings $p \mapsto h_p $, $p\mapsto \ell _p$  are
  Lipschitz continuous  on $\P(B)$. 
\end{corollary}
\cite{Gi4} proves also regularity properties of the mapping $p \mapsto h_p$ for random
walks on regular languages, which adapt, for instance, to the case of virtually free groups.

The best results of regularity to-date are due to P. Mathieu:
 \begin{proposition}[\cite{Ma2}] If $G$ is a hyperbolic group and   $B$ is finite and symmetric,  the mapping $p \mapsto h_p $ is $C^1$ on the set $\P_\s(B)$ of symmetric probability measures on $B$, endowed with its natural smooth structure. \end{proposition} 
Let  $\l \mapsto p_\l, \l \in [-\e, +\e]$ be a smooth curve in $\P(B)$. We write: 
 \begin{eqnarray*} &  &\lim_{\l \to 0 } \frac{h(p_\l) - h(p_0)}{\l} \\ \; &=&
   \lim_{\l \to 0 } \frac{1}{\l}    \sum _{g \in B} \left(\int_{\partial G}  \vf _{p_\l} (g, \xi )d\nu _{p_\l} (\xi)  p_\l (g) - \int_{\partial G}  \vf _{p_0} (g, \xi )d\nu _{p_0} (\xi)   p_0 (g) \right)\\ \; &=&  \lim _{\l \to  0 } \frac{1}{\l} \sum _{g \in B} \left(\int_{\partial G} ( \vf _{p_\l} (g, \xi )- \vf _{p_0} (g, \xi )) d\nu _{p_\l} (\xi) \right) p_\l (g) + \\
 \; &  & \quad + \sum _{g \in B} \left(\int_{\partial G}  \vf _{p_0} (g, \xi ))[ \lim _{\l \to  0 } \frac{1}{\l} ( d\nu _{p_\l}  - d\nu _{p_0})](\xi )  \right) p_\l (g) + \\
\;  &  & \quad  \quad + \sum _{g \in B} \left(\int_{\partial G}  \vf _{p_0} (g, \xi )) d\nu _{p_0}(\xi )  \right) [ \lim _{\l \to  0 } \frac{1}{\l} (p_\l(g) - p_0(g)] . \end{eqnarray*}
The third  line converges by definition. To prove that the second line converges, P. Mathieu observes that the Green metric $-\ln F_{p_0} (g,h) $ on $G$  satisfies (BA)  and a form of hyperbolicity that allows him to extend Proposition \ref{anal}. More precisely, he shows directly the differentiability of $\l \mapsto \int f(\xi)  d\nu _{p_\l} (\xi) $, for $f \in \K_\a$, and  gives a formula for the derivative. For the first line, P. Mathieu shows a general result for any non-amenable group $G$ and $p_\l $ with finite support. In our case, his result writes:
 \begin{proposition}[\cite{Ma2}] \label{pfft} Let $\l \mapsto p_\l, \l \in [-\e, +\e]$ be a smooth curve in $\P(B)$. Then, 
  $$ \lim _{\l \to  0 } \frac{1}{\l} \sum _{g \in B} \left(\int ( \vf _{p_\l} (g, \xi )- \vf _{p_0} (g, \xi )) d\nu _{p_\l} (\xi) \right) p_\l(g)  \; = \; 0 .$$
  \end{proposition}

 It is likely that the function $p \mapsto h_p$ has more regularity on $\P(B)$, but this is an open problem.
 
 \

 Another natural extension is towards more general families of probability
 measures on $G$. Proposition \ref{analfree} is valid for $p$ varying in finite
 dimensional affine subsets of $\{ p ; \sum _{g\in G} e^{\g|g|} p(g) <+\infty
 \}$ for some $\g>0 $ (see \cite {Le}). The other properties rest on Harnack
 inequality at infinity (see \cite {An}), which has been proven only for
 probability measures with finite support on hyperbolic groups.  Finally, let
 $\P^1(G) $ be the set of probabilities on $G$ satisfying $\sum_{g\in G} |g| p(g) <+\infty $ endowed with the topology of convergence on the functions which grow slower than $C|g|$ at infinity. The first  observation on this topic of regularity of the entropy is the fact that, if $G$ is hyperbolic, $p \mapsto h_p $ is continuous on $\P^1(G)$ (\cite {EK}).
 
 \

\end{document}